\documentclass[12pt]{article}

\usepackage{amsfonts}
\usepackage{amssymb}
\usepackage{amsthm}
\usepackage{amsmath}
\usepackage{enumerate}
\usepackage{hyperref}

\usepackage[usenames,dvipsnames] {color}

\usepackage{graphicx}

\def\IC{\mathbb{C}}
\def\ID{\mathbb{D}}

\def\IN{\mathbb{N}}

\def\m{\mathbb{M}}
\def\md{\mathbb{M}^{[d]}}

\def\mm{\mathbb{M}_m}
\def\mn{\mathbb{M}_n}
\def\mnd{\mathbb{M}_n^d}

\def\mmd{\mathbb{M}_m^d}
\def\id{{\rm id}}

\def\idn{{\rm id}_{\IC^n}}

\def\DD{{\bf D}}

\def\ot{\otimes}

\def\rpq{R_{pq}}

\def\Rpq{{\bf R}_{pq}}

\newtheorem{proposition}{Proposition}
\newtheorem{lemma}[proposition]{Lemma}
\newtheorem{theorem}[proposition]{Theorem}

\theoremstyle{definition}
\newtheorem{definition}[proposition]{Definition}
\theoremstyle{remark}
\newtheorem{question}{Question}
\newtheorem{example}[proposition]{Example}

\numberwithin{equation}{section}

\def\be{\begin{equation}}
\def\ee{\end{equation}}
\def\bd{\begin{definition}}
\def\ed{\end{definition}}
\def\O{\Omega}
\def\s{\sigma}
\def\oec{{}{\hspace*{\fill} $\lhd$} \vskip 5pt \par}

\begin{document}

\bibliographystyle{plain}

\title{NC Automorphisms of nc-bounded domains}
\author{John E. M\raise.45ex\hbox{c}Carthy
\thanks{Partially supported by National Science Foundation Grant  
DMS 1300280}
\and Richard Timoney\thanks{
Supported
by the Science Foundation Ireland under grant 11/RFP/MTH3187.
}}
\maketitle

\begin{abstract}
	We establish rigidity (or uniqueness) theorems for nc
	automorphisms which are natural extensions of clasical results
	of H.~Cartan and are improvements of recent results.
	We apply our results to nc-domains consisting of
	unit balls of rectangular matrices.
\end{abstract}

\section{Introduction}

Holomorphic automorphisms of domains in $\IC^d$ have been studied since the work of
H. and E. Cartan in the 1930's \cite{HCartab1930CRp718}, \cite{care35}.
A holomorphic function can be thought of as a generalized polynomial, 
and they can be evaluated not just on tuples of complex numbers, but also on tuples of
commuting matrices or commuting operators whose spectrum is in the domain of the function
\cite{tay70b}.
An nc-function (nc stands for non-commutative) is a generalization of a free polynomial, ({\em i.e.} a
polynomial in non-commuting variables), and it is natural to evaluate them on tuples of matrices or operators.

To describe nc-functions (following
\cite{kvv14} for instance),
we must first establish some notation.
Let $\mn$ denote the $n$-by-$n$ complex matrices, and $\mnd$ the $d$-tuples of
$n$-by-$n$ matrices. 
We shall let $\md $ denote the disjoint union $ \cup_{n=1}^\infty \mnd$.
Given  $x = (x^1, \ldots, x^d)$ in $\mnd$ and $y= (y^1, \ldots, y^d)$ in $\mmd$, by $x \oplus y$ we mean
the element $ (x^1 \oplus y^1, \ldots, x^d \oplus y^d)$ of $\m_{m+n}^d$. If $x \in \mnd$ and $s,t \in \mn$, by 
$sxt$ we mean $(s x^1 t, \dots, s x^dt)$.

A free polynomial $p$ in $d$ variables can be thought of as a function defined
on $\md$, and as such it has  the  following properties:

(i)  If $x $ is in
$\mn^d$, then $p(x) \in \mn$.

(ii) If $x$ and $y$ are in $\md$,
then $p(x \oplus y)  = p(x) \oplus p(y)$.

(iii) If $x \in \mn^d$ and $s \in \mn$ is invertible, then
$p (s^{-1} x s)  = s^{-1} p(x) s$.

\bd
\label{def:ncdomain}
An nc-set is a set  $\Omega \subseteq \md$
 such that  $\Omega_n := \Omega \cap \mnd$ is an open
set for each $n$, and such that  $\Omega$
 is closed with respect to direct sums
and joint unitary equivalence ({\em i.e.} for all $x \in \Omega_n$ and for all $u$ unitary in $\mn$,
we have $u^{-1} x u \in \Omega$). If an nc-set $\Omega$ 
has the property that $\O_n$ is connected for every $n$,
we shall call it an nc-domain.
\ed

An nc-function is a function on an nc-set that mimics the properties (i) - (iii) above of free polynomials.
\bd
\label{def:ncfunct}
An nc-function $f$ on an nc-set $\Omega$ is a function with the following three properties:

(i)  If $x $ is in
$\O_n$, then $f(x) \in \mn$ (we say $f$ is graded if this occurs).

(ii) If $x$ and $y$ are in $\O$,
then $f(x \oplus y)  = f(x) \oplus f(y)$.

(iii) If $x \in \O_n$,   $s \in \mn$ is invertible, and $s^{-1} x s \in \O$,
 then
$f (s^{-1} x s)  = s^{-1} f(x) s$.
\ed

An nc-map $\Phi$ on an nc-domain $\Omega \subseteq \md$ is a $d$-tuple of nc-functions.
If $\Phi$ is an nc-map on $\Omega$ that is also a bijection onto $\Omega$, we call it an nc-automorphism.

Bounded symmetric domains in $\IC^d$ have been characterized by
E. Cartan
\cite{care35}, and in the course of the proof automorphisms of such
domains were described.

 We are interested in the following questions about nc-automorphisms.

\begin{question}
\label{q1}  (Rigidity) If $\Phi$ is an nc-automorphism of $\Omega$, is it uniquely determined by its action
on $\Omega \cap \mnd$ for some fixed $n$?
\end{question}

\begin{question}
\label{q2}
(Extendibility) If $F: \Omega \cap \mnd \to \Omega \cap \mnd$ is a biholomorphic map which respects similarities, is there an nc-automorphism $\Phi: \Omega \to \Omega$ such that 
$\Phi |_{\Omega \cap \mnd} = F $?
\end{question}

\begin{question}
\label{q3}
What groups can arise as the automorphism group of an  nc-domain?
\end{question}

A set $\O \subset \md$ is called {\em nc-bounded} if for each $n$, there exists a constant $M_n$
such that
\[
\forall \ z \ \in \O_n, \quad \| z \| < M_n .
\]

In Theorem~\ref{propa} we show that one always has rigidity on nc-bounded domains that contain the origin.
For such domains, the possible automorphism groups are therefore no more than certain subgroups 
of the automorphism groups of bounded domains (and  we
substantially answer questions \ref{q1}, \ref{q2} and \ref{q3} for
circular bounded domains that contain the origin).
In Example \ref{exd2}, we show that many different domains can have the same automorphism groups.

In Section~\ref{secd}, we consider the Cartan domain of type I, the set $\rpq$ of $p$-by-$q$ contractive matrices.
The obvious nc-domain containing this, where numbers are replaced by $n$-by-$n$ matrices,
we call $\Rpq$. 
In Theorem~\ref{thme1}, we show that when $p \neq q$,  all automorphisms of $\rpq$ extend;
but when $p=q$, only those automorphisms that do not involve the transpose extend.

In this paper, we restrict our attention to nc-bounded domains, as the unbounded case is much more
complicated (see {\em e.g.} Example~\ref{ex3}).

This note continues work of Popescu in \cite{po06,po10,po11} and
of Helton, Klep, McCullough and Slinglend,
in \cite{hkms09}.

\section{Background on nc-functions}

The recent monograph \cite{kvv14} by D. Kaliuzhnyi-Verbovetskyi \& V. Vinnikov gives an introduction to nc-functions.
Unless an additional hypothesis of 
continuity (or boundedness) is added,
 nc-functions can behave badly.

\begin{example}
Let $d=1$, and  define a function $f$ on Jordan blocks by
 sending a Jordan block with $0$ eigenvalues to the zero matrix of the same size, and a Jordan block
 with non-zero eigenvalues to the identity matrix of that size. Extend $f$ by direct sums to any matrix in Jordan canonical form, and then by similarity to any matrix. The function $f$ is then an nc-function which is manifestly discontinuous.
 \end{example}
 
Let $\sigma$ denote the disjoint union topology on $\md$: a set $U$ is in $\sigma$ if and only if
$U \cap \mnd$ is open for every $n$. (This topology is called the
finitely open topology in \cite{kvv14}).

It was proved in \cite{amfree} that if an nc-function $f$ on an nc-set $\O$ is $\s$ locally bounded,
in the sense that 
\[
\forall \ z \ \in \O,\ \exists\ U \in \s \ {\rm s.t.\ }  z \in U \ {\rm and\ } f |_{\O \cap U} \ {\rm is\ bounded},
\]
then $f: \Omega \to \m^{[1]}$ is $\s$-$\s$ continuous, and in \cite{hkm11b}, it was shown that this 
in turn implied that $f$ was holomorphic, in the sense that
\be
\label{eqb1}
\forall n \in \IN,\ \forall z \in \O_n, \ \forall h \in \mnd, \quad \exists \
Df(z) [h] \ = \
\lim_{t \to 0} \frac{1}{t} [ f(z + th) - f(z) ] .
\ee

Putting these results together, we conclude
\begin{proposition}
\label{propb1} An nc-map $f$ into an nc-bounded domain   is automatically
$\s$-$\s$ continuous, and holomorphic in the sense of \eqref{eqb1}.
\end{proposition}

\section{Rigidity}
\label{secb}

A domain is called {\em circular} if it is invariant under multiplication by unimodular scalars.

The following lemmas are classical and due to H. Cartan.

\begin{lemma}[H. Cartan ({\cite[Th\'eor\`eme VII, p. 30]{car31}})]
	\label{lem:caruniqueness}
	Let $\mathcal{D} \subseteq \IC^d$ be a bounded domain, $z_0 \in
	\mathcal{D}$ and
	$\phi
	\colon \mathcal{D} \to \mathcal{D}$
	a biholomorphic automorphism with $\phi(z_0) = z_0$ and
	$\phi'(z_0) =
	I_n$. Then $\phi$ is the identity.
\end{lemma}

\begin{lemma}[{\cite{HCartab1930CRp718}, \cite[Th\'eor\`eme VI]{car31}}]
	\label{lem2}
	If $\mathcal{D}$ is a bounded circular domain 
 in $\IC^d$ containing $0$, and
	$F \colon \mathcal{D} \to \mathcal{D}$ is a biholomorphic
	automorphism of $\mathcal{D}$ with $F(0) =0$, then $F$
	is the restriction to $\mathcal{D}$ of an invertible linear
	map.
\end{lemma}

\begin{theorem}
\label{propa}
Let $\Omega$ be an nc-domain that is  nc-bounded.
Let $\Phi = (\Phi^1,  \ldots, \Phi^d)$ be an nc-automorphism of $\Omega$.
Suppose that for some $m \in \IN$, we have $0 \in \Omega \cap \mmd $ and
$(\Phi |_{ \Omega_m})(0) = 0$.
\begin{enumerate}[(i)]
	\item
		If in addition
$(\Phi |_{ \Omega_m})'(0)$ is the identity,
then  $\Phi$ is the identity on all of $\Omega$.

\item
	If instead we suppose also that
	$\Omega$ is a circular  nc-domain,
then there is an invertible linear map $F$ on $\IC^d$ such that
$\Phi(Z) = (F \otimes \id_n)(Z)$ for $Z \in \Omega_n$ (by which we mean
that each $d$-tuple 
$( \Phi(Z) )_{i,j}$
formed from the $(i,j)$ coordinates
of the $d$-tuple of
$n \times n$ matrices $\Phi(Z)$ is
given by $F (Z_{i,j})$
where $Z = (Z^1, \ldots, Z^d)$, and $Z_{i,j} = (Z^1_{i,j}, \ldots,
Z^d_{i,j})$ again denotes the $(i,j)$ coordinates).
\end{enumerate}
\end{theorem}

\begin{proof}
\begin{enumerate}[(i)]
	\item
		By Lemma~\ref{lem:caruniqueness}, $\Phi |_{ \Omega_m}$
		is the identity.

As $\Omega_m = \Omega \cap \mmd$ is open, there is some $\varepsilon >0$
such that if $z = (z^1, \dots, z^d) \in \mmd$ has each $\| z^j \| <
\varepsilon$, then $z \in \Omega_m$.

Now, fix a positive integer $k$.
For $Z^j \in \m_{km}$, we
write $\hat{Z}^j$ for the $d$-tuple 
$ (Z^1, \ldots, Z^d) \in \m_{km}^d$
that has $Z^i = 0$ when $i \neq j$, and whose $j^{\rm th}$ entry is $Z^j$.
If $Z^j$ is 
the
direct sum of $k$ matrices from $\mm$,
 and if $\| Z^j \| \leq \varepsilon$,
 then $\Phi(\hat{Z}^j) = \hat{Z}^j$, by the direct sum property of nc-maps. 
As this applies to $\zeta Z^j$ for $|\zeta| < 1$ ($\zeta \in \IC$),
it follows that the directional derivative
$((\Phi | _{\Omega_{km}})'(0))(\hat{Z}^j) = \hat{Z}^j$.
By linearity of the Fr\'echet derivative $(\Phi | _{\Omega_{km}})'(0)$
we may drop the restriction that $\|Z^j\| < \varepsilon$.
In particular the conclusion holds when $Z^j \in \m_{km}$ is a diagonal matrix.

By the chain rule and similarity invariance
\[
	(\Phi | _{\Omega_{km}}) (s^{-1} Z  s) = 
	s^{-1} (\Phi | _{\Omega_{km}}) ( Z  )  s 
\]
of the map $\Phi |
_{\Omega_{km}}$ (valid for all
sufficiently small $Z$ once $s$ is fixed),
we must have that $(\Phi | _{\Omega_{km}})'(0)$ has the
invariance property
\begin{equation}
	\label{eqn:siminv-deriv}
	((\Phi | _{\Omega_{km}})'(0)) (s^{-1} Z  s) = 
	s^{-1} (((\Phi | _{\Omega_{km}})'(0)) ( Z  ))  s 
\end{equation}
Choosing $Z^j$ diagonalisable and $s$ such that $s^{-1}Z^j s$ is diagonal
(\ref{eqn:siminv-deriv}) yields $s^{-1} (((\Phi | _{\Omega_{km}})'(0)) (
\hat{Z}^j  ))  s  = s^{-1} \hat{Z}^j s$ and hence 
$((\Phi | _{\Omega_{km}})'(0))(\hat{Z}^j) = \hat{Z}^j$ provided
$Z^j$
is diagonalisable. By density of the
diagonalisable matrices we can then make the same conclusion with $Z^j$
arbitrary and then linearity of $(\Phi | _{\Omega_{km}})'(0)$ forces it
to be the identity.

By Lemma~\ref{lem:caruniqueness} again, $\Phi |_{ \Omega_{km}}$
		must be the identity.

Now choose some $n$, not necessarily a multiple of $m$, such that 
$\Omega_n$ is non-empty. Let 
 $Z \in \Omega_n$. The direct sum of $m$ copies of $Z$ is in 
$ \Omega_{mn}$, and $\Phi (\oplus_{i=1}^m Z) = \oplus_{i=1}^m Z$ 
by the first part of the proof.
As $\Phi$ preserves direct sums, this means that $\Phi(Z) = Z$.

\item
By
Lemma~\ref{lem2}, we know $\Phi|_{\Omega_{m}}$ is linear,
If $m =1$ we take $F$ to be $\Phi|_{\Omega_{m}}$. For $m > 1$ we need
a brief argument to find $F$.

Similarity invariance
\[
	(\Phi | _{\Omega_{m}}) (s^{-1} Z  s) = 
	s^{-1} (\Phi | _{\Omega_{m}}) ( Z  )  s 
\]
(guaranteed by Definition~\ref{def:ncfunct}
to hold for $Z, s^{-1}Zs \in \Omega_{m}$)
must hold globally for $Z \in \mmd$ in view of linearity.
Let $E_{i,k}$ denote the standard matrix units in $\m_m$ and
choose $Z = (z^1 E_{1,1}, \ldots, z^d E_{1,1})$ (so that $Z$ is
supported on the 
$(1,1)$ entries).
Consider a block diagonal
$s =1 \oplus t$ with 1 in the $(1,1)$ entry but arbitrary invertible
$(m-1) \times (m-1)$ block $t$. Since the matrices that commute with all
such $1 \oplus t$ are those of the form $F \oplus \alpha I_{m-1}$ (for
scalars $F$ and $\alpha$)
we see that $\Phi|_{\Omega_{m}}
(Z)$ must take the form
\[
		(F^1(z) \oplus
		\alpha^1(z) I_{m-1},
\ldots, 
F^d(z) \oplus
\alpha^d(z) I_{m-1}),
\mbox{ with } z = (z^1, \ldots , z^d),
\]
for some scalar-valued linear $F^1, \ldots, F^d, \alpha^1,\ldots,
\alpha^d \colon \IC^d \to \IC$.

However, notice that 
(by Lemma~\ref{lem2}) we also know that
$\Phi|_{\Omega_{2m}}$ is linear, and by the direct sum property
$\Phi(Z \oplus 0) = \Phi(Z) \oplus 0$. The similarity argument applied
to $\m_{2m}^d$ forces $\alpha^j = 0$ ($1 \leq j \leq d$). Thus
\[
	\Phi |_{\Omega_m}(z^1 E_{1,1}, \ldots, z^d E_{1,1}) =
	(F^1(z) E_{1,1},
\ldots,\\
F^d(z) E_{1,1} )
\]
(with $ z = (z^1, \ldots, z^d)$).
Using similarity with $s$ a transposition allows us to conclude that the
same must hold for $E_{k,k}$ replacing $E_{1,1}$. Taking $s = I_n -
E_{i,k}$ for $i \neq k$ we have
$s^{-1} E_{ii} s  = E_{ii} - E_{ik}$, and together with linearity we
deduce the relation with $E_{k,k}$ replaced by $E_{i,k}$.
Clearly $F = (F^1, \ldots, F^d)  \colon \IC^d \to \IC^d$ must be
invertible (since $\Phi$ is) and we have the desired conclusion on
$\Omega_m$.

By
Lemma~\ref{lem2}, we know $\Phi|_{\Omega_{km}}$ is linear for each $k
\in \IN$.
For $Z^j \in \m_{km}$,
if $Z^j$ is diagonal we
must have $\Phi(\hat{Z}^j)$ of the required form. The similarity property and
density of the diagonalisable matrices in $\m_{km}$ allows us to extend
to arbitrary nonzero $Z^j$. Then by linearity this extends
to arbitrary $d$-tuples
$Z = (Z^1, \ldots, Z^d) \in \m_{km}^d$.

Finally if $\Omega_n$ is nonempty for some $n$ we can apply the result
just obtained for $\Omega_{nm}$ together with the direct sum property
for $Z \oplus 0 \oplus \cdots \oplus 0$ (where
$Z \in \m_n$ and we have $(m-1)$ zero
summands) to
establish the desired conclusion for $\Phi |_{\Omega_n}$.
\qedhere
\end{enumerate}
\end{proof}

Popescu's Cartan uniqueness results \cite[\S1]{po10} can be viewed
as similar in spirit for the case of special domains (row
contractions) to Theorem~\ref{propa}. In the case $m = 1$,
the result of \cite[Corollary 4.1 (2)]{hkm11b} is part
(i) while \cite[Theorem 21]{hkm12} implies (ii).

\begin{example}
\label{ex1}
The matrix polydisk. This is the set $$
\DD \ =\ \{ x \in \md : \| x^j \| < 1, 1 \leq j \leq d \}. $$
The set of automorphisms of $\DD$ is the set
\begin{equation}
\label{eqc1}
\{ \Phi(x) = \sigma \circ (m^1(x^1), \dots, m^d(x^d)) : \sigma \in {\bf S}_d, m^j \in {\rm Aut}(\ID) \} .
\end{equation}
Here ${\bf S}_d$ is the symmetric group on $d$ variables, and $ {\rm Aut}(\ID)$ is the M\"obius group of automorphisms of the disk. Each M\"obius transformation of the form
\[
m(z) \ = \ e^{i\theta} \frac{z - a}{1 -\bar a z} 
\]
extends to matrices in the obvious way:
\[
m(Z) \ = \ e^{i\theta} (Z - a I)(1 -\bar a Z)^{-1} .
\]
Indeed, by von Neumann's inequality, every $\Phi$ in \eqref{eqc1} extends to an automorphism
of $\DD$. That this comprises everything follows from observing that all automorphisms of
the polydisk $\ID^d = \DD \cap \m_1^d$ are of this form, and so by Theorem~\ref{propa} they
have a unique extension to higher levels. As they are invertible, they must be automorphisms.
\oec
\end{example}

\begin{example}
\label{ex3}
Theorem \ref{propa} fails if boundedness is dropped.
Consider, for example, the nc-set
\[
\Omega \ = \
\{ (x,y,z) \in \m^{[3]} : \| xy - yx \| < 1 \}.
\]
Let 
\[
\Phi(x,y,z) \ = \ (x, y, z + h(xy-yx) ),
\]
where $h : \IC \to \IC$ is any non-constant entire function mapping $0$ to $0$.
Then $\Phi$ is an automorphism, and $\Phi |_{\Omega \cap \m_1^3}$ is the identity,
but $\Phi$ is not the identity on level 2.
\oec
\end{example}

\section{Extendibility in $\Rpq$}
\label{secd}

Let $\rpq$ denote the $p$-by-$q$ matrices of norm less than $1$.
We shall extend this to an nc-domain in $\md$, where $d = pq$, by
\[
\Rpq \ :=\
\bigcup_{n=1}^\infty 
\{ (x^1, x^2, \dots, x^d) \in \mnd :
\left \| 
\begin{pmatrix}
x^1 & \dots & x^q \\
x^{q+1} & \dots & x^{2q} \\
\vdots \\
x^{(p-1)q +1} & \dots & x^{pq}
\end{pmatrix}
\right\| < 1 \} .
\]
 Let $\gamma: \mnd \to \m_{pq}^d$ be the map that takes $d$ matrices 
and makes them into a block $p$-by-$q$ matrix, filling in left to right and then top to bottom.
Then $\Rpq = \{ x : \| \gamma (x) \| < 1 \}$. When we speak of an nc automorphism of
$\Rpq$, strictly speaking we mean an nc automorphism of $\gamma^{-1} (\Rpq)$.

In the special case $q =1$, $R_{p1}$ is just the unit ball in $\IC^p$, and its automorphisms are well-known (see {\em e.g.} \cite[Thm. 2.2.5]{rud80}).
The set ${\bf R}_{1q}$, the row-contractions, 
was studied by G. Popescu in \cite{po91,po06} and \cite{po10}.

The automorphisms of $\rpq$ are given by a similar formula to the case of the ball. 
L. Harris showed \cite{Ha73} that  they are of the following form.

\begin{theorem}
\label{thm:ha}
[Harris]
Every holomorphic  automorphism of $\rpq$ is of the form
$L H_A$ where $L$ is a linear isometric automorphism of $\rpq$, 
 $A$ is an element of $\rpq$,
and 
\[
H_A(x) \ = \
( I_{p}  -  A A^* )^{-1/2}  
  (x  + A)
(I_{q} + A^*  x )^{-1}
( I_{q}  -   A^* A)^{1/2} . 
\] 
\end{theorem}

First, let us consider that automorphisms that map $0$ to $0$, which are the linear ones.
 K. Morita \cite{mo41} classified the linear isometries of $\rpq$, and the square case 
 differs from the rectangular case, because the transpose is an isometry.
\begin{theorem}
\label{thm:mor}
[Morita] If $p \ne q$, all linear automorphisms of $\rpq$ are of the form 
$x \mapsto U x V$, where $U$ is a $p$-by-$p$ unitary and $V$ is a $q$-by-$q$ unitary.
If $p = q$, the set of linear automorphisms consists of  $x \mapsto U x V$ and
$x \mapsto U x^t V$.
\end{theorem}

The map $x \mapsto UxV$ extends to the nc automorphism of $\Rpq$ given by
$Z \mapsto (\id \ot U) Z (\id \ot V)$; but the transpose does not extend.

\begin{lemma}
\label{leme1} If $p > 1$,
the map $x \mapsto x^t$ does not extend to an nc automorphism of ${\bf R}_{pp}$.
\end{lemma}

\begin{proof}
	Using Theorem~\ref{propa},
	if the transpose did extend, the extension would map
	$(X_{i,j}) \in \m_p(\m_n(\IC))$ to $(X_{j,i})$, and so
	this reduces to the well-known
	fact that the transpose map is not a complete isometry of
	$\m_p$.
\end{proof}

Let us turn now to $H_A$.
The map  $H_A$ extends
 to an nc map from 
$\Rpq$ to $\Rpq$ given by
\begin{equation}
\label{eqc6}
H_A(Z) \ = \
( I_{n,p}  - \id \ot A A^* )^{-1/2}  
  (Z  + \id \ot A)
(I_{n,q} +( \id \ot A^* ) Z  )^{-1}
( I_{n,q}  - \id \ot  A^* A)^{1/2}  
\end{equation}
Here, $\id$ means $\idn$, and
$I_{n,r}$ denotes $\id_{\IC^n \ot \IC^r}$.

A calculation shows that
\begin{eqnarray}
\nonumber
\lefteqn{\id - H_A (W)^* H_A(Z) \ = \
( I_{n,q}  - \id \ot  A^* A)^{1/2}  
( I_{n,q} + W^* (\id \ot A))^{-1}  } \\
&(I_{n,q} - W^* Z) 
 ( I_{n,q} + ( \id \ot A^*) Z )^{-1}
( I_{n,q}  - \id \ot  A^* A)^{1/2}  .
\label{eqc7}
\end{eqnarray}
Letting $W = Z$ proves that $H_A$ maps $\Rpq$ to $\Rpq$, and as $H_{-A}$ is the inverse
of $H_A$, it must be an automorphism.

Putting these results together, we get the following theorem.
The case $p=1$ was proved by Popescu \cite[Thm. 1.5]{po10}.
 The general case was
  proved by 
  Helton, Klep, McCullough and Slinglend \cite[Thm. 1.7]{hkms09},
  though their hypotheses are 
  stronger. The linear case was proved by D. Blecher and D. Hay
  \cite{bh03}.
\begin{theorem}
\label{thme1}
If $p \ne q$, then 
every automorphism of $\rpq$ extends uniquely to an automorphism of $\Rpq$.
If $p =q$, the automorphisms of the form $x \mapsto U H_A(x) V$ extend uniquely to
${\bf R}_{pp}$, and the automorphisms of the form
$x \mapsto U H_A(x)^t V$ (when $p>1$)
do not extend to automorphisms of ${\bf R}_{pp}$.
\end{theorem}

By a result of J. Ball and V. Bolotnikov \cite{babo04},
 (see also \cite{at03} and \cite{amy13a})
$H_A$ extends to an endomorphism of the {\em commuting} elements of $\gamma^{-1}(\Rpq)$
 if and only if there is some function $F$ so that 
\[
I - H_A (w)^* H_A(z)
\ = \
F(w)^* (I - w^* z) F(z)
\]
as a kernel on $\rpq$.
This is true, as \eqref{eqc7} shows. So the nc automorphisms of
$\{ x \in \gamma^{-1} (\Rpq) : x^i x^j = x^j x^i, \forall 1 \leq i,j \leq d \}$
are the same as the nc automorphisms of $\gamma^{-1} (\Rpq)$.
This phenomenon has also been explored in \cite{amfree}.

\begin{question} The automorphisms of $\Rpq$ are not transitive at any level $n \geq 2$. What can one say  about the orbits?
\end{question}

Theorem~\ref{thme1} can be extended slightly. For $S$ a
subset of $\IN$ that is closed under addition, let  
$\Rpq(S)$ be defined by
 $\Rpq(S) \cap \mnd$ is $\Rpq \cap \mnd$ if $n \in S$, and empty otherwise.

\begin{proposition}
\label{propd1}
Let $S$ be any non-empty sub-semigroup of $\IN$. Then the automorphisms of $\Rpq(S)$ are the
same as the automorphisms of $\Rpq$, and are uniquely determined by their action on any non-empty level.
\end{proposition}

\begin{example}
\label{exd2}
Extendibility can fail if the pieces of $\Omega$ at different levels are not somehow alike.
For example, let $d=1$, and $R > 1$.
Define $\Omega$ by $\Omega \cap \m_1 = \ID$, and $\Omega \cap \mn = \{ x : \| x \| < R \}$.
The automorphisms of $\Omega \cap \m_1$ are the M\"obius maps, but only multiplication by 
$e^{i\theta}$ extends to be an automorphism of $\Omega$.

But there are many other choices of nc-domain $\Omega \supset {\bf R}_{11}$ that have the same automorphism group.
For example, let $r_1 = 1$, and let $(r_n)_{n=1}^\infty$ be any non-decreasing sequence.
Define $\Omega$ by
\begin{equation}
	\label{eq:spectarldisk}
\Omega_n \ = \ \{ x \in \mn : \ \exists s \in \mn, {\ \rm with \ } \| s \| \| s^{-1} \| \leq r_n,\ {\rm and\ }
\| s^{-1} x s \| < 1 \} .
\end{equation}
Then $\O$ is an nc-domain, and its automorphism group is the set of M\"obius maps.
\oec
\end{example}

\begin{question}
\label{quest5}
Let $U \subset \IC^d$
be a bounded symmetric domain. Is there  an nc-domain $\O \subset \md$
such that
$\O_1 = U$ and such that the automorphism group of $\O$ equals the automorphism group
of $U$?
\end{question}

\bibliography{references}

\end{document}